\documentclass[11pt]{article}
\usepackage{geometry}
\usepackage{color}
\usepackage{float}
\usepackage{textcomp}
\usepackage{amsthm}
\usepackage{amsmath}
\usepackage{amssymb}

\usepackage{changes}

\usepackage{mathtools}

\newtheorem{theorem}{Theorem}[section]
\newtheorem{lemma}[theorem]{Lemma}

\newtheorem{proposition}[theorem]{Proposition}

\newtheorem{definition}[theorem]{Definition}

\DeclareMathOperator{\argmin}{argmin}

\DeclareMathOperator{\dom}{dom}

\newcommand{\R}{\mathbb{R}}

\newcommand{\inner}[2]{\langle{#1},{#2}\rangle}
\newcommand{\Inner}[2]{\left\langle{#1},{#2}\right\rangle}

\newcommand{\cP}{{\cal{P}}}

\newcommand{\vgap}{\vspace{.1in}}

\newcommand{\bi}{\begin{itemize}}
\newcommand{\ei}{\end{itemize}}
\newcommand{\ba}{\begin{array}}
\newcommand{\ea}{\end{array}}

\begin{document}

\title{Iteration-complexity of a  Jacobi-type non-Euclidean  ADMM  for  multi-block linearly constrained  nonconvex programs}

\author{
    Jefferson G. Melo 
     \thanks{Instituto de Matem\'atica e Estat\'istica, Universidade Federal de Goi\'as, Campus II- Caixa
    Postal 131, CEP 74001-970, Goi\^ania-GO, Brazil. (E-mail:  {\tt jefferson@ufg.br}).  The work of this author was partially   supported by  CNPq Grants 406250/2013-8, 444134/2014-0 and 406975/2016-7.} 
    \and
    Renato D.C. Monteiro
    \thanks{School of Industrial and Systems
    Engineering, Georgia Institute of
    Technology, Atlanta, GA, 30332-0205.
    (email: {\tt monteiro@isye.gatech.edu}). The work of this author
    was partially supported by NSF Grant CMMI-1300221.}
}

\date{May 13, 2017}

\maketitle

\begin{abstract}
This paper establishes the  iteration-complexity of a Jacobi-type non-Euclidean proximal alternating direction method of multipliers (ADMM)
for solving multi-block linearly constrained nonconvex programs.  The subproblems of this ADMM
variant  can be solved in parallel and hence the method has great potential to solve large scale multi-block linearly constrained  nonconvex programs.
Moreover, our analysis allows the Lagrange multiplier to be updated with a relaxation parameter in the interval $(0,2)$.
\\

 \end{abstract}

\pagestyle{plain}
\section{Introduction} \label{sec:int}
This paper considers  the following linearly constrained optimization problem 
\begin{equation} \label{prob:opt1}
\min  \left\{ \sum_{i=1}^p f_i(x_i) : \sum_{i=1}^pA_ix_i =b, \; x_i\in \R^{n_i}, i=1,\ldots,p\right\}
\end{equation}
where $f_i: \R^{n_i} \to (-\infty,\infty]$, $i=1,\ldots,p$, are proper lower semicontinuous functions,
$A_i\in\R^{d \times n_i}$, $i=1,\ldots, p$, and  $b \in \R^d$.

Optimization problems such as  \eqref{prob:opt1} appear in many important applications such as  distributed matrix factorization, distributed clustering, sparse zero variance discriminant analysis, tensor decomposition, matrix completion, and asset allocation (see, e.g., \cite{Ames2016,5712153,Liavas,Wen2013,Xu2012,Zhang2014}).
Recently, some variants  of the alternating direction method of multipliers (ADMM) have been successfully  applied to solve some instances of the previous problem despite the lack of convexity.

In this paper we analyze the Jacobi-type proximal ADMM for solving~\eqref{prob:opt1}, which recursively computes a sequence
$\{(x_1^k,\cdots,x_p^k,\lambda^k)\}$ as
\begin{align}\label{ADMMclass}
x_i^k = \argmin_{x_i} \left \{ {\mathcal L}_\beta(x^{k-1}_1,\ldots,x_{i-1}^{k-1},x_i,x_{i+1}^{k-1},\ldots,x_p^{k-1},\lambda^{k-1}) +(dw_i)(x_i,x_i^{k-1})\right\},\quad i=1,\ldots,p,\end{align}
$$
\lambda^k = \lambda^{k-1}-\theta\beta\left(\sum_{i=1}^pA_ix_i^k-b\right)
$$
where $\beta>0$ is a penalty parameter, $\theta>0$ is a relaxation parameter,  $dw_i$ is a Bregman distance,  and 
 \begin{equation}\label{augLagIntro}
{\mathcal L}_\beta(x_1,\ldots,x_p,\lambda):= \sum_{i=1}^p  f_i(x_i) - \left\langle \lambda,\sum_{i=1}^p  A_ix_i-b\right\rangle +\frac{\beta}{2}\left\|\sum_{i=1}^pA_ix_i-b\right\|^2
\end{equation}
is  the augmented Lagrangian  function for problem \eqref{prob:opt1}.
An important feature of  this ADMM variant is that the subproblems \eqref{ADMMclass}  can be solved in parallel and hence the method has great potential to solve large scale multi-block linearly constrained  nonconvex programs.
Under the assumption that  $A_p$ is full row rank and $f_p: \R^{n_p}\to \R$ is a differentiable function whose gradient is Lipschitz continuous, 
 we establish an ${\cal O}(\rho^{-2})$ iteration-complexity bound for the Jacobi-type ADMM \eqref{ADMMclass}
to obtain
 $(x_1,\ldots,x_p, \lambda,r_1,\ldots,r_{p-1})$ satisfying
\begin{align}
r_i &\in \partial f_i(x_1,\ldots,x_p) - A_i^*\lambda, \quad  \;i=1,\cdots,p-1,\\[3mm]
\max& \left\{\left\|\sum_{i=1}^pA_ix_i -b\right \| , \left\|r_1\right\|,\cdots,\|r_{p-1}\| , \left\|\nabla f_p(x_p) - A_p^* \lambda\right\|\right\} \le \rho
\end{align}
where $\partial f_i$ denotes the limiting subdifferential (see for example \cite{Mordu2006,VariaAna}).

We briefly discuss in this paragraph the development of ADMM in the convex setting.
The standard ADMM  (i.e., where $p=2$, $w_i \equiv 0$ for $i=1,\ldots,2$ and $x_2^k$ is obtained as above but with $x_1^{k-1}$ replaced by $x_1^k$)
was  introduced in \cite{0352.65034,0368.65053}
and its complexity analysis  was first carried out in \cite{monteiro2010iteration}.
Since then  several papers  have obtained iteration-complexity results  for  various ADMM variants (see for example \cite{Cui,MJR2,Gu2015,He2,HeLinear,Deng1,GADMM2015,MJR,Hager,Lin,LanADMM}).
Multiblock ADMM variants have also been extensively studied (see for example 
\cite{Han2012,Hong2017,Lin2015Siam,Lin2015ch,Lin2016,JacobiHe2016,WotaoinJacobi,paralleladm,LogJacobiADMM}). In particular, papers \cite{JacobiHe2016,WotaoinJacobi,paralleladm,LogJacobiADMM} study the convergence and/or complexity of Jacobi-type ADMM variants.

Recently, there have been a lot of interest on the study of ADMM variants for nonconvex problems (see, e.g., \cite{ADMM_KL,Hong2016,Jiang2016,SplitMet_NonConv,Multi-blockBregman,BregmanADMM,wotao2015,Stepzisenoncon,RMJ2017a,MBProxNonconvRenJeff,JacobiADMM1}).
Papers \cite{ADMM_KL,SplitMet_NonConv,Multi-blockBregman,BregmanADMM,wotao2015,Stepzisenoncon} establish
convergence of the generated sequence to a stationary point of \eqref{prob:opt1} under conditions which guarantee that a certain potential function associated with the augmented Lagrangian
\eqref{augLagIntro} satisfies the
Kurdyka-Lojasiewicz property.
However, these papers do not study the iteration complexity of the proximal  ADMM although their theoretical analysis are
generally half-way towards accomplishing such goal. Paper \cite{Hong2016}  analyzes the convergence of variants of the
ADMM for solving nonconvex consensus and sharing problems and establishes the iteration
complexity of  ADMM for the consensus problem.
Paper \cite{Jiang2016} studies the iteration-complexity of two linearized variants of the multiblock proximal ADMM
applied to a more general problem than \eqref{prob:opt1} where a coupling term is also present in its objective function.
Paper \cite{RMJ2017a} studies the iteration-complexity of a proximal ADMM  for the two block optimization problem, i.e., $p=2$,
and the relaxation parameter $\theta$ is arbitrarily chosen in the interval $(0,2)$, contrary to the previous related literature where this parameter is considered as one or at most $(\sqrt{5}+1)/2$.  Paper \cite{MBProxNonconvRenJeff}  analyzes the iteration-complexity of a multi-block proximal ADMM via a general linearization scheme.
Finally, while the authors were in the process of finalizing this paper, they have learned of the recent paper \cite{JacobiADMM1} which
studies the asymptotic convergence of a Jacobi-type linearized ADMM for solving non-convex problems.
The latter paper though does not deal with the issue of iteration-complexity and considers the case of $\theta=1$ only.

Our paper is organized as follows.
Subsection~\ref{sec:bas} contains some notation and basic results
used in the paper.
Section~\ref{ProximalADMM} describes our assumptions and contains two subsections. Subsection~\ref{NEPJ-ADMM}  introduces the concept of distance generating functions (and its corresponding Bregman distances) considered in this paper, and formally states the non-Euclidean Jacobi-type ADMM. Section~\ref{subset:convanalysis} is devoted to the convergence rate analysis of the latter method. Our main convergence rate result is in this subsection (Theorem \ref{maintheo}).
The appendix contains proofs of some results stated in the paper.

\subsection{Notation and basic results}
\label{sec:bas}

The domain of a function
 $f :\mathbb{R}^s\to (-\infty,\infty]$ is the set   $\dom f := \{x\in \R^{s} : f(x) < +\infty\}$. 
Moreover, $f$ is said to be proper if  $f(x) < \infty$ for some $x \in \R^{s}$.  

\begin{lemma}\label{le:147} 
Let $S \in \R^{n\times p}$ be a non-zero matrix and let $\sigma^+_S$ denote  the smallest positive eigenvalue of $SS^*$. Then, 
for every $u \in \mathbb{R}^p$, there holds
\[
\|\cP_{S^*}(u)\|\leq \frac{1}{\sqrt{\sigma^+_S}}\|Su\|.
\]
\end{lemma}

We next recall some definitions and results of  subdifferential calculus~\cite{Mordu2006,VariaAna}.

\begin{definition}
Let $h: \R^{s} \to(-\infty,\infty]$ be a proper lower semi-continuous function.
\begin{enumerate}
\item[(i)] The Fr\'echet subdifferential of $h$ at $x\in \dom h$, denoted by $\hat{\partial} h(x)$, is the set of all elements $u \in \R^{s}$ satisfying
$$\liminf_{y\neq x\; y\to x} \frac{h(y)-h(x)-\inner{u}{y-x}}{\|y-x\|}\geq0.$$
When $x\notin \dom h$, we set $\hat{\partial} h(x)=\emptyset$.
\item[(ii)] The limiting subdifferential  of $h$ at $x \in dom\, h$, denoted by ${\partial} h(x)$, is defined as 
\[
\partial h (x)=\{u\in \R^{s}:\exists \,  x^k \to x, h(x^k)\to h(x), u^k \in \hat{\partial} h(x^k), \;\mbox{with}\; u^k \to u \}.
\]
\item[(iii)]  A critical (or stationary) point of $h$ is a point $x \in \dom h$ satisfying $0\in \partial h(x)$.
\end{enumerate}
\end{definition}  
The following result presents some properties of the limiting subdifferential.
 
\begin{proposition}\label{prop:subdiffcalc} Let $h: \R^{s} \to (-\infty,\infty]$ be a proper lower semi-continuous function.
\begin{enumerate}
\item[(a)]  If $x \in R^s$ is a local minimizer of $h$, then $0\in  \partial h(x)$;
\item[(b)] If $g: \R^{s} \to \R$ is a continuously differentiable function, then $\partial(h + g)(x) = \partial h(x) + \nabla g(x)$.
\end{enumerate}
\end{proposition}


\section{Jacobi-type non-Euclidean proximal ADMM and its convergence rate}\label{ProximalADMM}
 We start by recalling  the definition of  critical points of \eqref{prob:opt1}.
 \begin{definition}\label{def:criticalpoint}
 An element  $ (x_1^*,\ldots,x_p^*, \lambda^*) \in \R^{n_1}\times\ldots \times \R^{n_p}\times \R^d$ is a critical point of problem~\eqref{prob:opt1}
if
\[ 0\in  \partial f_i(x_i^*)-A_i^*\lambda^*,\quad i=1,\ldots,p, 
\quad  \sum_{i=1}^pA_ix_i^*=b.
\]
\end{definition}
Under some mild conditions, it can be shown that if $(x_1^*,\ldots,x_p^*)$ is a global minimum of \eqref{prob:opt1}, then there exists
$\lambda^*$ such that $ (x_1^*,\ldots,x_p^*, \lambda^*)$ is a critical point of \eqref{prob:opt1}.

The augmented Lagrangian associated with problem \eqref{prob:opt1} and with penalty parameter $\beta>0$  is defined as
\begin{equation}\label{lagrangian2}
{\mathcal L}_\beta(x_1,\ldots,x_p,\lambda):=\sum_{i=1}^pf_i(x_i)-\left\langle \lambda,\sum_{i=1}^pA_ix_i-b\right\rangle+\frac{\beta}{2}\left\|\sum_{i=1}^pA_ix_i-b\right\|^2.
\end{equation}

We assume that  problem~\eqref{prob:opt1} satisfies the following  set of conditions:
\begin{itemize}
\item[{\bf(A0)}] The functions $f_i$, $i=1,\ldots,p-1$,  are   proper lower semicontinuous;
\item[{\bf(A1)}] $A_p \ne 0$ and $ {\rm Im} (A_p) \supset \{b\} \cup {\rm Im} (A_1) \cup\ldots \cup {\rm Im} (A_{p-1})$;

\item[{\bf(A2)}]
$f_p: \R^{n_p} \to \R$ is differentiable with gradient $L_p-$Lipschitz continuous.

\item[{\bf(A3)}] there exists $\bar \beta \ge 0$ such that $v(\bar \beta)>-\infty$ where
\[
v(\beta) :=\inf_{(x_1,\ldots,x_p)} \left \{\sum_{i=1}^pf_i(x_i) +\frac{\beta}{2} \left \|\sum_{i=1}^pA_ix_i -b \right\|^2 \right\} \quad \forall \beta \in \R. \label{assump:beta_inf}
\]
\end{itemize}

\subsection{The non-Euclidean proximal Jacobi ADMM}\label{NEPJ-ADMM}

In this subsection, we introduce a class of distance generating functions (and its corresponding Bregman distances)
which is suitable for our study. We also formally describe the non-Euclidean proximal Jacobi ADMM for solving problem~\eqref{prob:opt1}.

\begin{definition}\label{defw0}
For given set $Z\subset \R^s$ and scalars $m \le M$, we let $\mathcal{D}_Z(m,M)$ denote the class of real-valued
functions  $w$ which are differentiable on $Z$
and satisfy
\begin{align}\label{strongly}
w(z')-w(z) - \inner{\nabla w(z)}{z'-z} \geq \frac{{m}}{2}\|z-z'\|^2 \quad \forall  z,z' \in  Z,\\[2mm]
\label{a1}
\|\nabla w(z)-\nabla w(z')\|\leq {M}\|z-z'\| \quad \forall z, z' \in Z.
\end{align}
A function $w \in \mathcal{D}_Z(m,M)$ with $m\geq0$  is referred to as a distance generating function and its associated
Bregman distance $dw: \R^s \times Z \to \mathbb{R}$ is defined as
\begin{equation}\label{def_d}
(dw)(z';z) := w(z')-w(z)-\langle \nabla w(z),z'-z\rangle  \quad \forall (z', z) \in \R^s \times Z.
\end{equation}
\end{definition}

For every $z \in Z$,  the function $(dw)(\,\cdot\,;z)$ will be denoted by $(dw)_{z}$ so that
\[
(dw)_{z}(z')=(dw)(z';z) \quad \forall (z',z) \in \R^s \times Z.
\]
Clearly,
\begin{align}
\nabla (dw)_{z}(z') &= - \nabla (dw)_{z'}(z) = \nabla w(z') - \nabla w(z) \quad \forall z, z' \in Z, \label{grad-d}
\end{align}


We now state the non-Euclidean proximal Jacobi ADMM based on the class of distance generating functions introduced in
Definition \ref{defw0}.
In its statement and in some technical results, we denote the block of variables $(x_1,\ldots,x_{i-1})$ simply by 
$x_{<i}$ and the block of variables $(x_{i+1},\ldots,x_p)$ simply by 
$x_{>i}$. Hence, the whole  vector $(x_1,\ldots,x_p)$ can also be denoted as
$(x_{<i},x_i,x_{>i})$ when there is a need to emphasize the $i$-th  block. For convenience, we also extend the above for notation for $i=1$ and $i=p$. Hence,
$(x_{<1},x_1,x_{>1})=(x_{<p},x_p,x_{>p})=(x_1,\ldots,x_p)$.
 \\
\hrule
\noindent
\\
{\bf Non-Euclidean Proximal Jacobi ADMM (NEPJ-ADMM)}
\\
\hrule
\begin{itemize}
\item[(0)]Define $Z_i := \dom f_i$ for $i=1,\ldots,p,$  and let $\bar{\beta}$ be as in {\bf (A3)}. Let
 an initial point $(x_1^0,\ldots,x_p^0,\lambda^0) \in Z_1 \times \ldots \times Z_p\times  \R^d$.
Choose scalars $\alpha>0$,  $\beta\geq \bar{\beta}$,  $M_i \ge m_i>0$, $i=1,\ldots,p$, and  a stepsize parameter $\theta\in (0,2)$  such that   
\begin{align}\label{assump:delta1}
 \delta_i&:=\frac{m_i}{4}-\left(\frac{p-2+\alpha}{2}+\frac{2\gamma_\theta(p+1)}{\sigma_{A_p}^+} \| A_p^*\|^2\right) \beta\max_{1\leq l\leq p-1}\|A_l\|^2 >0, \quad i=1,\ldots,p-1\nonumber\\[5mm]
\delta_p&:=\frac{m_p}{4}-\left(\frac{\beta(p-1)\|A_p\|^2}{2\alpha}+\frac{ \gamma_\theta(p+1) (L_p^2 + 2M_p^2 ) }{\beta\sigma_{A_p}^+}\right)>0
\end{align}
where $\sigma_{A_p}$ (resp., $\sigma_{A_p}^+$) denotes the smallest  eigenvalue  (resp., positive eigenvalue) of $A_p^*A_p$, and $\gamma_\theta$ is given by
\begin{equation}\label{def:gamma}
\gamma_\theta:=\frac{\theta}{(1-|\theta-1|)^2}.
\end{equation}
Set $k=1$ and go to step 1.
\item[(1)]  For each $i=1,\ldots,p$, choose $w_i^k\in {\mathcal{D}_{Z_i}(m_i,M_i)}$ and compute an optimal solution $x_i^k \in \R^{n_i}$  of
\begin{equation} \label{sub:s}
\min_{x_i\in \R^{n_i}} \left \{ {\mathcal L}_\beta(x_{< i}^{k-1},x_i,x_{>i}^{k-1},\lambda^{k-1})+(d w^k_i)_{x_i^{k-1}}(x_i)\right\}.
\end{equation}
\item[(2)] Set 
\begin{equation}\label{eq:lambda}
\lambda^k = \lambda^{k-1}-\theta\beta\left[\sum_{i=1}^pA_ix_i^k-b\right],
\end{equation}
 $k \leftarrow k+1$, and go to step~(1).
\end{itemize}
{\bf end}
\\
\hrule
\vgap

Some comments about the NEPJ-ADMM  are in order. First, it is always possible to choose the constants $m_i$, i=1,\ldots,p,  sufficiently large  so as to
guarantee that $\delta_i$, i=1,\ldots,p, are strictly positive. 
Second, one of the main features of NEPJ-ADMM is that its subproblems \eqref{sub:s} are completely independent of one another.
As a result, they can all be solved in parallel which shows the potential of NEPJ-ADMM as a suitable ADMM variant to solve 
 large instance of \eqref{prob:opt1}. 
Third, as in the papers \cite{RMJ2017a,MBProxNonconvRenJeff}, NEPJ-ADMM allows the choice of a relaxation parameter $\theta \in (0,2)$.

 \subsection{Convergence Rate Analysis of the NEPJ-ADMM}\label{subset:convanalysis}
 This subsection is dedicated to the convergence rate analysis of the NEPJ-ADMM.

We first present some technical lemmas which are useful to prove our main result (Theorem~\ref{maintheo}).
To simplify the notation,  we denote by  $x^k$ the vector $(x_1^k,\ldots,x_p^k)$ generated by the NEPJ-ADMM.

 \begin{lemma} \label{pr:aux}
Consider the sequence $\{(x^k,\lambda^k)\}$  generated by the NEPJ-ADMM.  For every $k \ge 1$, define
\begin{equation}\label{def:lambdahat}
\hat\lambda^k:=\lambda^{k-1}-\beta\left(\sum_{i=1}^pA_ix_i^k-b\right)
\end{equation}
and
\begin{equation}
R_i^k:= - \sum_{j=1,j\neq i}^p\beta A_i^*A_j\Delta x_j^k+ \Delta w_i^k,\qquad i=1,\ldots,p, \label{def:R1R2}
\end{equation}
where
\begin{equation} \label{eq:Delta-def}
\Delta x_i^k := x_i^k - x_i^{k-1},\qquad  \Delta w_i^k    := \nabla w_i^k   (x_i^k)-\nabla w_i^k   (x_i^{k-1}),\qquad  i=1,\ldots,p.
\end{equation}
Then, for every $k \ge 1$, we have:
\begin{align}
 0&\in \partial f_i(x_i^k)-A_i^* \hat \lambda^k +R_{i}^k\quad i=1,\ldots,p\label{aux.211}\\[2mm]
0&=  \left[\sum_{i=1}^pA_ix_i^k-b \right]+\frac{1}{\theta\beta}\Delta \lambda^k\label{aux.32}
\end{align}
where $\Delta \lambda^k:=\lambda^k-\lambda^{k-1}$.
\end{lemma}
\begin{proof}
The optimality conditions (see Proposition~\ref{prop:subdiffcalc}) for \eqref{sub:s} imply that
\[
0 \in \partial f_i(x_i^k)- A_i^*\left[{\lambda}_{k-1}-\beta\left(A_ix_i^k+\sum_{j=1,j\neq i}^pA_jx_j^{k-1}-b\right)\right]+\Delta w_i^k,\quad i=1,\ldots,p.
\]
This  relation combined with \eqref{def:lambdahat} and \eqref{def:R1R2} immediately yield \eqref{aux.211}.
Relation \eqref{aux.32} follows directly from \eqref{eq:lambda}.
\end{proof}
Next result presents a recursive relation involving  the displacements $\Delta \lambda^k$ and $\Delta \lambda^{k-1}$.
\begin{lemma} \label{pr:aux-new}
Consider the sequence $\{(x^k,\lambda^k)\}$ generated by the NEPJ-ADMM and define
 \begin{equation}\label{def:p0}
R_p^0= A_p^*\lambda^0 -\nabla f_p(x_p^0),\quad  \Delta \lambda^0=0.
 \end{equation}
Then, for every $k \ge 1$, we have
\begin{equation}\label{eq:B*lambda}
 A_p^*\Delta\lambda^k=(1-\theta)A_p^*\Delta\lambda^{k-1}+\theta u^k, 
 \end{equation}
where 
\begin{equation}\label{def:uk}
u^k:=\Delta f_p^k +  \Delta R_p^k , \quad \Delta f_p^k:= \nabla f_p^k(x_p^k)-\nabla f_p^k(x_p^{k-1}), \quad \Delta R_p^k:= R_p^k  - R_p^{k-1}\quad \forall k \ge 1, 
\end{equation} 
 $ \Delta \lambda^k$ and $R_p^k$ are as in Lemma~\ref{pr:aux}.
\end{lemma}
\begin{proof}
From \eqref{def:lambdahat} and \eqref{aux.32}, we obtain the following relation
$$
\lambda^k=(1-\theta)\lambda^{k-1}+\theta \hat\lambda^k,\quad \forall k \ge 1.
$$
Using this relation and  \eqref{aux.211} with $i=p$, we have
\begin{equation}\label{eq:auxlem1}
 A_p^* \lambda^k=(1-\theta)A_p^*\lambda^{k-1}+ \theta [ \nabla f_p(x_p^k) +  R_p^k ], \quad \forall k \ge 1.
\end{equation}
Hence, in view of \eqref{def:uk},  relation \eqref{eq:B*lambda} holds for every $k \ge 2$.
Now, note that \eqref{def:p0} is equivalent to  $\nabla f_p(x_p^0)+R_p^0=   A_p^*\lambda^0$. This relation combined with
\eqref{def:uk} and \eqref{eq:auxlem1}, both with $k=1$, yield
\begin{align*}
A_p^* \Delta \lambda^1 &= - \theta A_p^*\lambda^0 +  \theta \left [\nabla f_p(y^1) + R_p^1\right ]\\[2mm]
&= - \theta A_p^*\lambda^0 +  \theta \left[\nabla f_p(x_p^0) + R_p^0+u^1\right]\\[2mm]
&=- \theta A_p^*\lambda^0 +\theta A_p^*\lambda^0+\theta u^1=\theta u^1.
\end{align*}
Hence,  in view of $\Delta \lambda^0=0$,  relation \eqref{eq:B*lambda} also holds for $k=1$.
\end{proof}

Next we consider an auxiliary result to be used to compare consecutive terms of the sequence  $\{{\mathcal L}_\beta(x^k,\lambda^k)\}$. See the comments immediately before the NEPJ-ADMM about the notation used hereafter.
\begin{lemma}\label{LastauxLem} For every $y^0=(y^0_1,\ldots,y^0_p), \ y=(y_1,\ldots,y_p) \in
 \dom f_1 \times\ldots\times \dom f_p$,  $\lambda \in \R^d$
and  $i=2,\ldots,p$, we have
\begin{align*}
{\mathcal L}_\beta (y_{<i},y_i,y_{>i}^0,\lambda) - {\mathcal L}_\beta (y_{<i},y_i^0,y_{>i}^0,\lambda) &=
{\mathcal L}_\beta (y_{< i}^0,y_i,y_{>i}^0,\lambda) - {\mathcal L}_\beta (y_{<i}^0,y_i^0,y_{>i}^0,\lambda) \\
& \ \ + \beta \sum_{j=1}^{i-1} \langle A_i\Delta y_i,A_j\Delta y_j\rangle.
\end{align*}
\end{lemma}
\begin{proof}
It is easy to see thatf the gradient of the function
\begin{equation}\label{eq:auxquadfunc}
y_{<i} \mapsto {\mathcal L}_\beta (y_{<i},y_i,y_{>i}^0) - {\mathcal L}_\beta (y_{<i},y_i^0,y_{>i}^0) 
\end{equation}
is given by
\[
\beta [ A_1 \cdots A_{i-1} ]^* A_i \Delta y_i
\]
and  its Hessian equal to zero everywhere in $\dom f_1 \times \ldots \times \dom f_{i-1}$. 
Hence, the  function given in \eqref{eq:auxquadfunc} is affine. The conclusion of the lemma now follows by noting that
\[
\Inner{ \left[ A_1 \cdots A_{i-1} \right]^* A_i \Delta y_i}{\Delta y_{<i}} =
\sum_{j=1}^{i-1} \langle A_i\Delta y_i,A_j\Delta y_j\rangle.
\]
\end{proof}

The next result  compares consecutive terms of the sequence  $\{{\mathcal L}_\beta(x^k,\lambda^k)\}$.

\begin{lemma} \label{lem:auxdecL} For every $k\geq 1$, we have 
\begin{equation*}
{\mathcal L}_\beta(x^k,\lambda^k)- {\mathcal L}_\beta(x^{k-1},\lambda^{k-1})\leq \sum_{1\le j<i\leq p}\beta\langle A_i\Delta x_i^k,A_j\Delta x_j^ k\rangle -\sum_{i=1}^{p}\frac{m_i}{2}\|\Delta x_i^k\|^2+\frac{1}{\theta\beta}\|\Delta \lambda^k\|^2.
\end{equation*}
\end{lemma}

\begin{proof}
First note that  \eqref{sub:s} together with the fact that $w_i^k\in {\mathcal{D}_{Z_i}(m_i,M_i)}$ imply that

\begin{equation*}\label{eq1:Lem_descLag}
{\mathcal L}_\beta(x_{<i}^{k-1},x_i^{k},x_{>i}^{k-1},\lambda^{k-1})- {\mathcal L}_\beta(x^{k-1},\lambda^{k-1})\leq -m_i\|\Delta x_i^k\|^2/2, \qquad i=1,\ldots,p.
\end{equation*}
\\
Hence, using Lemma~\ref{LastauxLem} with  $y^0=x^{k-1}$, $y=x^k$ and $\lambda=\lambda^{k-1}$, we see that
\begin{align*}
 {\mathcal L}_\beta (x_{< i}^k,x_i^k,x_{>i}^{k-1},\lambda^{k-1}) &-{\mathcal L}_\beta (x^{k}_{<i},x_i^{k-1},x_{> i}^{k-1},\lambda^{k-1}) \\
&=
{\mathcal L}_\beta (x_{< i}^{k-1},x_i^k,x_{>i}^{k-1},\lambda^{k-1})-{\mathcal L}_\beta (x^{k-1},\lambda^{k-1})
+ \beta \sum_{j=1}^{i-1} \langle A_i\Delta x^k_i,A_j\Delta x^k_j\rangle \\
&\le  - \frac{m_i}{2}\|\Delta x_i^k\|^2+\beta \sum_{j=1}^{i-1} \langle A_i\Delta x^k_i,A_j\Delta x^k_j\rangle.
\end{align*}
Hence
\begin{align}\label{eqauxmainlemma}
 {\mathcal L}_\beta(x^{k},\lambda^{k-1})-{\mathcal L}_\beta(x^{k-1},\lambda^{k-1})&=\sum_{i=1}^p\left[{\mathcal L}_\beta (x_{< i}^k,x_i^k,x_{>i}^{k-1},\lambda^{k-1}) -{\mathcal L}_\beta (x^{k}_{<i},x_i^{k-1},x_{> i}^{k-1},\lambda^{k-1})\right]\nonumber \\
 &\le  - \sum_{i=1}^p \frac{m_i}{2}\|\Delta x_i^k\|^2+\beta \sum_{i=2}^p\sum_{j=1}^{i-1} \langle A_i\Delta x^k_i,A_j\Delta x^k_j\rangle.
\end{align}
On the other hand, due to $\Delta \lambda^k =\lambda^k-\lambda^{k-1}$ and  \eqref{eq:lambda}, we have
$$
{\mathcal L}_\beta(x^{k},\lambda^k)-{\mathcal L}_\beta(x^k,\lambda^{k-1})=- \Inner{ \lambda^k-\lambda^{k-1}}{\sum_{i=1}^pA_ix_i^k-b}= \frac{1}{\beta\theta}\|\Delta \lambda^k\|^2.
$$
To conclude the proof, just add the last relation and  \eqref{eqauxmainlemma}.
\end{proof}

Lemma~\ref{lem:auxdecL} is essential to show that a certain sequence $\{\hat {\mathcal L}_k\}$ associated to $\{{\mathcal L}_\beta (x^k,\lambda^k)\}$ is monotonically decreasing.
This sequence  is defined as

\begin{align}
\hat {\mathcal L}_k := {\mathcal L}_\beta (x^k,\lambda^k) +\eta_k \qquad \forall k\geq 0 \label{def:Lhat},
\end{align}
where 
\begin{align}
\eta_0&:= \frac{m_p}{4M_p^2}\|A_p^*\lambda^0 -\nabla f_p(x_p^0)\|^2\label{def:eta_00}\\
\eta_k &:=\sum_{i=1}^{p}\frac{m_i}{4}\|\Delta x_i^k\|^2+\frac{c_1}{2}\|A_p^*\Delta\lambda^k\|^2\qquad \forall k\geq 1, \label{def:eta}
 \end{align}

\begin{equation}\label{def:c1} 
 c_1:=\frac{2|\theta-1|}{\beta\theta(1-|\theta-1|)\sigma^+_B}\geq0.
\end{equation}

Before establishing the monotonicity property of the  sequence \{$\hat {\mathcal L}_k\}$, we first  present an upper bound on $\hat {\mathcal L}_k - \hat {\mathcal L}_{k-1}$ in terms of some quantities
related to $\Delta x_1^k,\ldots,\Delta x_p^k$, and  $\Delta \lambda^k$.

\begin{lemma} For any $k\geq 1$, there holds
\begin{equation}\label{des:Ldec3}
\hat {\mathcal L}_k - \hat {\mathcal L}_{k-1} \leq \sum_{i=1}^{p-1}\left(\frac{(p-2+\alpha)\beta\|A_i\|^2}{2}-\frac{m_i}{4}\right)\|\Delta x_i^k\|^2+\Theta^k_\lambda +  \Theta_{p}^k
\end{equation}
where 
\begin{equation}\label{def:thetalambda}
\Theta^k_\lambda := 
\frac1{\beta\theta}\|\Delta \lambda^k\|^2 + \frac{c_1}2 \left(  \|A_p^*\Delta\lambda^k\|^2 - \|A_p^*\Delta\lambda^{k-1}\|^2 \right) 
\end{equation}

\begin{equation}\label{def:thetaxp}
\Theta_{p}^k:=\left(\frac{(p-1)\beta\|A_p\|^2}{2\alpha}-\frac{m_p}{4}\right)\left(\|\Delta x_p^k\|^2+\|\Delta x_p^{k-1}\|^2\right)
\end{equation}
and $\Delta \lambda^0=0$, $\Delta x_p^0=R_p^0/M_p$ (see Lemma~\ref{pr:aux-new}).
\end{lemma}
\begin{proof}
From Lemma~\ref{lem:auxdecL} and definitions of $\hat {\mathcal L}_k$ and $\Theta_\lambda^k$, we obtain 
\begin{align}
\hat {\mathcal L}_k - \hat {\mathcal L}_{k-1} & \le \sum_{1\le j<i< p}\beta\| A_i\Delta x_i^k\|\|A_j\Delta x_j^ k\| +\sum_{i=1}^{p-1}\beta\| A_i\Delta x_i^k\|\|A_p\Delta x_p^ k\| -\sum_{i=1}^{p-1}\frac{m_i}{2}\|\Delta x_i^k\|^2+ \Theta^k_\lambda\\
&- \frac{m_p}{4}\left(\|\Delta x_p^k\|^2 + \|\Delta x_p^{k-1}\|^2\right)\\
&\leq \sum_{1\le j<i< p}\left(\frac{\beta}{2}\| A_i\Delta x_i^k\|^2+\frac{\beta}{2}\|A_j \Delta x_j^k\|^2\right) +\sum_{i=1}^{p-1}\frac{\alpha\beta}{2}\|A_i\Delta x_i^k\|^2+\frac{(p-1)\beta}{2\alpha}\|A_p\Delta x_p^k\|^2 \\
&-\sum_{i=1}^{p-1}\frac{m_i}{2}\|\Delta x_i^k\|^2+ \Theta^k_\lambda-\frac{m_p}{4}\left(\|\Delta x_p^k\|^2 + \|\Delta x_p^{k-1}\|^2\right) \\
&\leq \sum_{i=1}^{p-1}\left(\frac{(p-2+\alpha)\beta\|A_i\|^2}{2}-\frac{m_i}{2}\right)\|\Delta x_i^k\|^2+\Theta^k_\lambda + \Theta_{p}^k
\end{align}
where  the inequalities are due  to Cauchy-Schwarz inequality and by using the relation $2s_1s_2\leq 
ts_1^2+(1/t)s_2^2$,  $s_1,s_2\in \R$ for $t=1$ and $t=\alpha$, respectively. 
 
\end{proof}

The next  result compares  $\Theta^k_\lambda$ with $\|u_k\|$, defined in \eqref{def:thetalambda} and  \eqref{def:uk}, respectively,  and provides an upper bound for both elements in terms of $(\Delta x_1^k,\ldots, \Delta x_p^k)$.

\begin{lemma}\label{lem:des_theta1uk}
Consider $\Theta^k_\lambda$ as in \eqref{def:thetalambda} and $u^k$ as in \eqref{def:uk}. Then, 
\begin{align}
\Theta^k_\lambda &\leq \frac{\gamma_\theta}{\beta\sigma_{A_p}^+} \|u_k\|^2\label{eq:theta1uk}\\ 
&\le \frac{\gamma_\theta(p+1)}{\beta\sigma_{A_p}^+} \left[\sum_{j=1}^{p-1}\beta^2\| A_p^*A_j\|^2(\|\Delta x_j^k\|+\|\Delta x_j^{k-1}\|)^2+\left(L_p^2+ M_p^2\right) \left( \|\Delta x_p^k\| +\|\Delta x_p^{k-1}\| \right)^2 \right].\nonumber
\end{align}
where $\gamma_\theta$ is as in  \eqref{def:gamma} and  $\Delta x_i^0=0, i=1,\ldots, p-1,$ and $\Delta x_p^0=R_p^0/M_p$ (see Lemma~\ref{pr:aux-new}).
\end{lemma}
\begin{proof}
The proof of this lemma is given in Appendix~\ref{proofLemmathetauk}.
\end{proof}

The next proposition shows, in particular, that  the sequence $\{\hat {\mathcal L}_k\}$ is  decreasing and  bounded below.

\begin{proposition}\label{prop:decL} Let $\Delta x_i^0=0, i=1,\ldots, p-1,$ and $\Delta x_p^0=R_p^0/M_p$.
Then, the following statements hold:
\begin{itemize}
\item[(a)] for every $k \ge 1$,
\[
\hat {\mathcal L}_k - \hat {\mathcal L}_{k-1} \le  -\sum_{i=1}^{p}\delta_i(\|\Delta x_i^k\|^2+\|\Delta x_i^{k-1}\|^2);
\]
\item[(b)]
the sequence $\{\hat {\mathcal L}_k\}$ given in \eqref{def:Lhat} satisfies $\hat {\mathcal L}_k \geq v(\beta)$ for every $k\geq 0$; 
\item[(c)] for every $k \ge 1$,
\[
\sum_{j=1}^k\sum_{i=1}^p \delta_i(\|\Delta x_i^j\|^2+\|\Delta x_i^{j-1}\|^2)
\leq \hat{\mathcal L}_0-v(\beta)
\]
\end{itemize}
where  $v(\beta)$ and $\delta_i$  are as in ${\bf (A3)}$ and \eqref{assump:delta1}, respectively.
\end{proposition}

\proof 
(a)  It follows from  \eqref{des:Ldec3}, Lemma~\ref{lem:des_theta1uk},  \eqref{assump:delta1}  and \eqref{def:thetaxp} that
\begin{align*}
\hat {\mathcal L}_k - \hat {\mathcal L}_{k-1} &\leq -\sum_{i=1}^{p-1}\left[\frac{m_i}{4}-\left(\frac{p-2+\alpha}{2}+\frac{2\gamma_\theta(p+1)}{\sigma_{A_p}^+} \| A_p^*\|^2\right) \beta\max_{1\leq j\leq p-1}\|A_i\|^2\right]\left(\|\Delta x_i^k\|^2+\|\Delta x_i^{k-1}\|^2\right)\\
 &-\left[\frac{m_p}{4}-\left(\frac{\beta(p-1)\|A_p\|^2}{2\alpha}+\frac{ \gamma_\theta(p+1) (L_p^2 + 2M_p^2 ) }{\beta\sigma_{A_p}^+}\right)\right]\left(\|\Delta x_p^k\|^2+\|\Delta x_p^{k-1}\|^2\right)  \\[2mm]
                                      &=-\sum_{i=1}^p\delta_i\left(\|\Delta x_i^k\|^2+\|\Delta x_i^{k-1}\|^2\right),
\end{align*}
proving (a). The proof of (b) is given Appendix~\ref{proof_item(b)}.
 The proof of (c) follows immediately from (a) and (b).
\endproof
Next proposition presents some convergence rate bounds for the displacements $\Delta x_i^k$, $i=1,\ldots,p$, and $\Delta \lambda^k$ in terms of some initial parameters.  Our main result will follow easily from this proposition, due to the fact that the residual generated by $(x^k,\hat \lambda^k)$ in order to satisfy the Lagrangian system \eqref{def:criticalpoint}  (see Lemma~\ref{pr:aux}) can be controlled by these displacements.
\begin{proposition}\label{lem:sumDeltaxylambda} 
Let $\delta_i$, i=1,\ldots,p, be as in   \eqref{assump:delta1} and  define
 \begin{equation}\label{def:deltalambda}
\delta_\lambda:= \left[\frac{\theta\gamma_\theta (p+1)}{\sigma_{A_p}^+\displaystyle\min_{1\leq l \leq p} \delta_l}\left(2\beta^2\| A_p^*\|^2\max_{1\leq l\leq p-1}\|A_l\|^2+ L_p^2+2M_p^2\right)\right]^{-1}
\end{equation}  
where $\Delta {\mathcal L}_0:=\hat{\mathcal L}_0-v(\beta)$ {\rm (see \eqref{def:Lhat} and ({\bf A3}))}.
Then, for every $k \ge 1$, we have
\begin{equation}
\sum_{j=1}^k\left\{\left[ \sum_{i=1}^p \delta_i\left(\|\Delta x_i^j\|^2+\|\Delta x_i^{j-1}\|^2\right)\right] +\delta_\lambda\|\Delta \lambda^j\|^2 \right\}\leq 2\Delta {\mathcal L}_0\label{des:deltay&x}
\end{equation}
and there exists $j\leq k$ such that
\begin{equation}
\|\Delta x_i^j\|\leq \sqrt{\frac{2\Delta {\mathcal L}_0}{k\delta_i}}, \;i=1,\ldots,p, \quad
\|\Delta \lambda^j\|\leq \sqrt{\frac{2\Delta {\mathcal L}_0}{k\delta_\lambda}}.\label{lem:bound_deltaxylambda}
\end{equation}
\end{proposition}
\begin{proof}
It follows from Proposition~\ref{prop:decL}(c) that
\begin{equation}\label{eq:auxdeltayBy}
 \quad \sum_{j=1}^k \sum_{i=1}^{p}\left(\|\Delta x_i^j\|^2+\|\Delta x_i^{j-1}\|^2\right)\leq  \frac{\Delta {\mathcal L}_0}{\displaystyle\min_{1\leq i \leq p} \delta_i}
\end{equation}
and that in order to prove  \eqref{des:deltay&x},  it suffices to show that  
\begin{equation}\label{des:deltaLambda}
\sum_{j=1}^k \|\Delta \lambda^j\|^2 \leq\frac{\Delta {\mathcal L}_0}{\delta_\lambda}.
\end{equation}
Then, in the remaining part of the proof we will show that  \eqref{des:deltaLambda} holds. By rewriting   \eqref{def:thetalambda}, we have
$$
\|\Delta \lambda^k\|^2= \beta\theta\left[ \frac{c_1}2 \left(  \|A_p^*\Delta\lambda^{k-1}\|^2 - \|A_p^*\Delta\lambda^k\|^2\right) + \Theta^k_\lambda \right]  \qquad \forall k\geq1.
$$
Hence, due to $\Delta \lambda^0=0$ and  Lemma~\ref{lem:des_theta1uk},  we obtain
\begin{align*}
\sum_{j=1}^k \|\Delta \lambda^j\|^2 
&\leq \beta\theta \sum_{j=1}^k\Theta^j_\lambda \leq \frac{\theta\gamma_\theta}{\sigma_{A_p}^+}\sum_{j=1}^k\|u^j\|^2\\
&\leq  \frac{\theta\gamma_\theta(p+1)}{\sigma_{A_p}^+} \left[2\beta^2\| A_p^*\|^2\max_{1\leq l\leq p-1}\|A_l\|^2\sum_{j=1}^k\sum_{i=1}^{p-1}\left(\|\Delta x_i^j\|^2+\|\Delta x_i^{j-1}\|^2\right)\right]\\
&+\frac{\theta\gamma_\theta(p+1)}{\sigma_{A_p}^+} (L_p^2+2M_p^2)\sum_{j=1}^k\left( \|\Delta x_p^j\|^2 +\|\Delta x_p^{j-1}\|^2 \right).\\
&\leq \frac{\theta\gamma_\theta (p+1)}{\sigma_{A_p}^+}\left(2\beta^2\| A_p^*\|^2\max_{1\leq l\leq p-1}\|A_l\|^2+ L_p^2+2M_p^2\right)\frac{\Delta {\mathcal L}_0}{\displaystyle\min_{1\leq i \leq p} \delta_i}\\
\end{align*}
where the fourth inequality is due to  \eqref{eq:auxdeltayBy}.
It is now to verify that  the previous estimate and  \eqref{def:deltalambda} imply \eqref{des:deltaLambda}, which in turn
implies \eqref{des:deltay&x}. 
\end{proof}

We now present  the main convergence rate result for the NEPJ-ADMM. Its main conclusion is that the NEPJ-ADMM  generates an element  $(\bar{x}_1,\ldots,\bar{x}_p,\bar{\lambda})$ which satisfies the optimality conditions of Definition~\ref{def:criticalpoint} within an error of $\mathcal{O}(1/\sqrt{k})$.

\begin{theorem}\label{maintheo} Let 
$\Delta {\mathcal L}_0  :={\mathcal L}_\beta(x^0,\lambda^0)- v (\beta) + \eta_0$
where $\eta_0$  and $v(\beta)$ are as in \eqref{def:eta_00} and {\bf (A3)}, respectively. Let  $\hat{\lambda}^k$ and $R_i^k$, $i=1,\ldots,p$, be as in   \eqref{def:lambdahat} and \eqref{def:R1R2}, respectively. Consider $\delta_i$, $i=1,\ldots,p$, as  in  \eqref{assump:delta1} and let
$\delta_{\lambda}$ be as in \eqref{def:deltalambda}.
Then, the following statements hold:
\begin{itemize}
\item[a)] $\Delta {\mathcal L}_0 \geq 0;$
\item[b)]  for every $k \ge 1$, 
\begin{align*}
 0&\in \partial f_i(x_i^k)-A_i^* \hat \lambda^k +R_{i}^k \qquad i=1,\ldots,p,
\end{align*}
and there exists $j\leq k$ such that
\begin{align*}
&\|R_i^j\|\leq \left[\sum_{l=1,l\neq i}^p\beta \|A_i^*A_l\|+ M_i\right]\sqrt{\frac{2\Delta {\mathcal L}_0}{k\displaystyle\min_{1\leq l \leq p} \delta_l}},\qquad i=1,\ldots,p,\\[2mm]
&\left\|\sum_{i=1}^pA_p x_i^j-b\right\|\leq \frac{1}{\beta\theta}\sqrt{\frac{2\Delta {\mathcal L}_0}{k\delta_\lambda }}.
\end{align*}
\end{itemize}
\end{theorem}

\begin{proof}
(a) holds due to Proposition~\ref{prop:decL}(c).  Lemma~\ref{pr:aux} shows that the first statement of (b) holds.  Now, it follows from   \eqref{def:R1R2}, \eqref{aux.32} and  the fact that $w_i^k\in {\mathcal{D}_{Z_i}(m_i,M_i)}$, $i=1,\ldots,p$, that
\begin{align*}
&\|R_i^k\|\leq \sum_{l=1,l\neq i}^p\beta \|A_i^*A_l|\|\Delta x_l^k\|+ M_i\|\Delta x_i^k\|,\qquad i=1,\ldots,p, \\
 &\left\|\sum_{i=1}^pA_i x_i^k-b\right\|=\frac{1}{\beta\theta}\|\Delta\lambda^k\|.
\end{align*}
Hence, to end the proof, just combine the above relations with \eqref{lem:bound_deltaxylambda}.
\end{proof}
\appendix
\section{Proof of Lemma~\ref{lem:des_theta1uk}}\label{proofLemmathetauk}
Let us first prove the first inequality \eqref{eq:theta1uk}.
Assumption {\bf (A1)} clearly  implies that $$\Delta \lambda_k=-\beta\theta\left(\sum_{i=1}^pA_ix_i-b\right)\in  {\rm Im}(A_p).$$
Hence, it follows from  Lemma~\ref{le:147} that
$$
\|\Delta \lambda_k\|=\|\cP_{A_p}(\Delta \lambda_k)\|\leq \frac{1}{\sqrt{\sigma_{A_p}^+}}\|A_p^*\Delta \lambda_k\|.
$$
Thus, in view of \eqref{eq:B*lambda} and \eqref{def:thetalambda}, we have
\begin{align*}
\Theta^k_\lambda &\leq \frac{1}{\beta\theta\sigma_{A_p}^+}\|A_p^*\Delta\lambda_k\|^2+\frac{c_1}{2}\left(\|A_p^*\Delta\lambda_k\|^2-\|A_p^*\Delta\lambda_{k-1}\|^2\right)\\
&=\left(\frac{1}{\beta\theta\sigma_{A_p}^+}+\frac{c_1}{2}\right)\|(1-\theta)A_p^*\Delta\lambda_{k-1}+\theta u_k\|^2-\frac{c_1}{2}\|A_p^*\Delta\lambda_{k-1}\|^2.
\end{align*}
Note that if $\theta=1$, then \eqref{def:c1} implies that $c_1=0$ and the above inequality proves the first inequality of the lemma.
We will now prove the first inequality of the lemma for the case in which $\theta \ne 1$.
The previous inequality together with the relation $\|s_1+s_2\|^2\leq (1+t)\|s_1\|^2+(1+1/t)\|s_2\|^2$
which holds for every $s_1,s_2\in \R^m$ and $t>0$
yield
\small{
\begin{align*}
\Theta^k_\lambda&\leq\left(\frac{1}{\beta\theta\sigma_{A_p}^+}+\frac{c_1}{2}\right)\left[(1+t)(\theta-1)^2\|A_p^*\Delta\lambda_{k-1}\|^2+\left(1+\frac{1}{t}\right)\theta^2\|u_k\|^2\right]-\frac{c_1}{2}\|A_p^*\Delta\lambda_{k-1}\|^2\\
&=\left[\left(\frac{1}{\beta\theta\sigma_{A_p}^+}+\frac{c_1}{2}\right)(1+t)(\theta-1)^2-\frac{c_1}{2}\right]\|A_p^*\Delta\lambda_{k-1}\|^2+\left(\frac{1}{\beta\theta\sigma_{A_p}^+}+\frac{c_1}{2}\right)\left(1+\frac{1}{t}\right)\theta^2\|u_k\|^2\\
&=\tiny\left\{\frac{(1+t)(\theta-1)^2}{\beta\theta\sigma_{A_p}^+}-\left[1-(1+t)(\theta-1)^2\right]\frac{c_1}{2}\right\}
\|A_p^*\Delta\lambda_{k-1}\|^2+\left(\frac{1}{\beta\theta\sigma_{A_p}^+}+\frac{c_1}{2}\right)\left(1+\frac{1}{t}\right)\theta^2\|u_k\|^2.
\end{align*}
}
Using the above expression with $t= -1+1/|\theta-1|$ and noting that $t>0$ in view of the assumption that $\theta \in (0,2)$, we
conclude that
\begin{align*}
\Theta^k_\lambda &\leq \left[\frac{1}{\beta\theta\sigma_{A_p}^+}|\theta-1|-\left(1-|\theta-1| \right)\frac{c_1}{2}\right]\|A_p^*\Delta\lambda_{k-1}\|^2
+\left(\frac{1}{\beta\theta\sigma_{A_p}^+}+\frac{c_1}{2}\right)\frac{\theta^2}{1-|\theta-1|}\|u_k\|^2 \\
&=\frac{1}{\beta\theta\sigma_{A_p}^+}\left(1+\frac{|\theta-1|}{1-|\theta-1|}\right)\frac{\theta^2}{1-|\theta-1|}\|u_k\|^2
\end{align*}
where the last equality is due to \eqref{def:c1}.
Hence, in view of \eqref{def:gamma},  the first inequality of the lemma is proved.

 We now prove the second  inequality in \eqref{eq:theta1uk}. 
Due to  $R_p^0=M_p\Delta x_p^0$, $w_p^k\in {\mathcal{D}_{\R^{n_p}}(m_p,M_p)}$,  assumption {\bf (A2)}, and relation \eqref{def:uk}, we obtain 
\begin{align*}
\|u^k\|^2 & =
\left\|\Delta f_p^k+ \Delta R_p^k\right\|^2\\[2mm]
              &\leq \left[ L_p\|\Delta x_p^k\|+ \sum_{j=1}^{p-1}\beta\| A_p^*A_j\|\left(\|\Delta x_j^k\|+\|\Delta x_j^{k-1}\|\right)+ M_p\left(\|\Delta x_p^k\| +\|\Delta x_p^{k-1}\|\right)\right]^2\\
               &  \leq (p+1) \left[ L_p^2\|\Delta x_p^k\|^2+  \sum_{j=1}^{p-1}\beta^2\| A_p^*A_j\|^2\left(\|\Delta x_j^k\|+\|\Delta x_j^{k-1}\|\right)^2+M_p^2 \left( \|\Delta x_p^k\| +\|\Delta x_p^{k-1}\| \right)^2 \right]
\end{align*}
where  the  inequalities follow from the triangle inequality for norms, definition of $\Delta R_p^k$ in \eqref{def:uk}, and 
the relation $\left(\sum_{i=1}^{l}s_i\right)^2\leq l \left(\sum_{i=1}^l s_i^2\right)$ for $s_i\in \R$, $i=1,\ldots,l$. Hence the proof of  Lemma~\ref{lem:des_theta1uk} follows.
\qed

\section{Proof of Lemma~\ref{prop:decL}(b)}\label{proof_item(b)}
 Note that  due to (a), we just need  to  prove the statement of (b) for $k\geq 1$. Hence,
assume by contradiction that there exists an index $k_0 \ge 0$ such that $ \hat {\mathcal L}_{k_0+1} <v(\beta)$.
Since by (a), $\{\hat {\mathcal L}_k\}$ is decreasing, we  obtain
\[
\sum_{k=1}^j (\hat {\mathcal L}_k - v(\beta)) \le
\sum_{k=1}^{k_0} (\hat {\mathcal L}_k-v(\beta)) + (j-k_0) (\hat {\mathcal L}_{k_0+1}-v(\beta) ) \quad \forall j > k_0,
\]
which implies that 
\[
\lim_{j \to \infty} \sum_{k=1}^j (\hat {\mathcal L}_k - v(\beta)) = - \infty.
\]
On the other hand, it follows from \eqref{lagrangian2}, \eqref{eq:lambda}, \eqref{def:Lhat} and {\bf(A3)} that
\begin{align*}
\hat {\mathcal L}_k& =\mathcal{L}_\beta(x^k,\lambda^k)+\eta_k\ge \mathcal{L}_\beta(x^k,\lambda^k) \\
&= \sum_{i=1}^pf_i(x_i^k)+ \frac{ \beta}{2} \left\| \sum_{i=1}^pA_i x_i^k -b \right\|^2 + \frac{1}{\beta\theta}\inner{ \lambda^k}{\lambda^k-\lambda^{k-1}}\\[2mm]
&\geq v(\beta)+ \frac{1}{2\beta\theta}\left(\|\lambda^k\|^2-\|\lambda^{k-1}\|^2+\|\lambda^k-\lambda^{k-1}\|^2\right)\geq
v(\beta)+ \frac{1}{2\beta\theta}\left(\|\lambda^k\|^2-\|\lambda^{k-1}\|^2\right)
\end{align*}
and hence that
\[
\sum_{k=1}^j ( \hat {\mathcal L}_k -v(\beta) ) \ge \frac{1}{2\beta\theta}\left(\|\lambda^j\|^2-\|\lambda^0\|^2\right)\geq -\frac{1}{2\beta\theta}\|\lambda^0\|^2 \quad \forall j \ge 1,
\]
which yields the desired contradiction.


\begin{thebibliography}{10}

\bibitem{Ames2016}
B.~P.~W. Ames and M.~Hong.
\newblock Alternating direction method of multipliers for penalized
  zero-variance discriminant analysis.
\newblock {\em Comput. Optim. Appl.}, 64(3):725--754, 2016.

\bibitem{Cui}
Y.~Cui, X.~Li, D.~Sun, and K.~C. Toh.
\newblock On the convergence properties of a majorized {A}{D}{M}{M} for
  linearly constrained convex optimization problems with coupled objective
  functions.
\newblock {\em J. Optim. Theory Appl.}, 169(3):1013--1041, 2016.

\bibitem{Deng1}
W.~Deng and W.~Yin.
\newblock On the global and linear convergence of the generalized alternating
  direction method of multipliers.
\newblock {\em J. Sci. Comput.}, pages 1--28, 2015.

\bibitem{WotaoinJacobi}
Wei Deng, Ming-Jun Lai, Zhimin Peng, and Wotao Yin.
\newblock Parallel multi-block admm with o(1 / k) convergence.
\newblock {\em Journal of Scientific Computing}, 71(2):712--736, 2017.

\bibitem{GADMM2015}
E.~X. Fang, B.~He, H.~Liu, and X.~Yuan.
\newblock Generalized alternating direction method of multipliers: new
  theoretical insights and applications.
\newblock {\em Math. Prog. Comp.}, 7(2):149--187, 2015.

\bibitem{5712153}
P.~A. Forero, A.~Cano, and G.~B. Giannakis.
\newblock Distributed clustering using wireless sensor networks.
\newblock {\em IEEE J. Selected Topics Signal Process.}, 5(4):707--724, 2011.

\bibitem{0352.65034}
D.~Gabay and B.~Mercier.
\newblock {A dual algorithm for the solution of nonlinear variational problems
  via finite element approximation.}
\newblock {\em Comput. Math. Appl.}, 2:17--40, 1976.

\bibitem{0368.65053}
R.~Glowinski and A.~Marroco.
\newblock Sur l'approximation, par \'el\'ements finis d'ordre un, et la
  r\'esolution, par penalisation-dualit\'e, d'une classe de probl\`emes de
  dirichlet non lin\'eaires.
\newblock 1975.

\bibitem{MJR2}
M.~L.~N. Gon{\c{c}}alves, J.~G. Melo, and R.~D.~C. Monteiro.
\newblock Extending the ergodic convergence rate of the proximal {A}{D}{M}{M}.
\newblock {\em Arxiv preprint: https://arxiv.org/abs/1611.02903}, 2016.

\bibitem{RMJ2017a}
M.~L.~N. Goncalves, J.~G. Melo, and R.~D.~C. Monteiro.
\newblock Convergence rate bounds for a proximal {A}{D}{M}{M} with
  over-relaxation stepsize parameter for solving nonconvex linearly constrained
  problems.
\newblock {\em Arxiv Preprint: https://arxiv.org/abs/1702.01850}, 2017.

\bibitem{MJR}
M.~L.~N. Gon{\c{c}}alves, J.~G. Melo, and R.~D.~C. Monteiro.
\newblock Improved pointwise iteration-complexity of a regularized {A}{D}{M}{M}
  and of a regularized non-euclidean {H}{P}{E} framework.
\newblock {\em SIAM J. Optim.}, 27(1):379--407, 2017.

\bibitem{Gu2015}
Y.~Gu, B.~Jiang, and H.~Deren.
\newblock A semi-proximal-based strictly contractive {P}eaceman-{R}achford
  splitting method.
\newblock {\em Arxiv preprint: https://arxiv.org/abs/1506.02221}, 2015.

\bibitem{ADMM_KL}
K.~Guo, D.~R. Han, and T.~T. Wu.
\newblock Convergence of alternating direction method for minimizing sum of two
  nonconvex functions with linear constraints.
\newblock {\em Int. J. Comput. Math. ``DOI: 10.1080/00207160.2016.1227432"},
  2016.

\bibitem{Hager}
W.~W. Hager, M.~Yashtini, and H.~Zhang.
\newblock An ${O}(1/k)$ convergence rate for the variable stepsize {B}regman
  operator splitting algorithm.
\newblock {\em SIAM J. Numer. Anal.}, 54(3):1535--1556, 2016.

\bibitem{Han2012}
D.~Han and X.~Yuan.
\newblock A note on the alternating direction method of multipliers.
\newblock {\em J. Optim. Theory Appl.}, 155(1):227--238, 2012.

\bibitem{He2}
B.~He, F.~Ma, and X.~Yuan.
\newblock On the step size of symmetric alternating directions method of
  multipliers.
\newblock {\em Preprint: http://www.optimization-online.org}, 2015.

\bibitem{JacobiHe2016}
B.~He, H-K. Xu, and X.~Yuan.
\newblock On the proximal jacobian decomposition of alm for multiple-block
  separable convex minimization problems and its relationship to admm.
\newblock {\em Journal of Scientific Computing}, 66(3):1204--1217, 2016.

\bibitem{HeLinear}
B.~He and X.~Yuan.
\newblock On the $\mathcal{O}(1/n)$ convergence rate of the
  {D}ouglas-{€"R}achford alternating direction method.
\newblock {\em SIAM Journal on Numer. Anal.}, 50(2):700--709, 2012.

\bibitem{Hong2017}
M.~Hong and Z.-Q. Luo.
\newblock On the linear convergence of the alternating direction method of
  multipliers.
\newblock {\em Math. Programming}, 162(1):165--199, 2017.

\bibitem{Hong2016}
M.~Hong, Z.-Q. Luo, and M.~Razaviyayn.
\newblock Convergence analysis of alternating direction method of multipliers
  for a family of nonconvex problems.
\newblock {\em SIAM J. Optim.}, 26(1):337--364, 2016.

\bibitem{Jiang2016}
B.~Jiang, T.~Lin, S.~Ma, and S.~Zhang.
\newblock Structured nonconvex and nonsmooth optimization: algorithms and
  iteration complexity analysis.
\newblock {\em Arxiv Preprint: https://arxiv.org/abs/1605.02408}, 2016.

\bibitem{SplitMet_NonConv}
G.~Li and T.~K. Pong.
\newblock Global convergence of splitting methods for nonconvex composite
  optimization.
\newblock {\em SIAM J. Optim.}, 25(4):2434--2460, 2015.

\bibitem{LogJacobiADMM}
M.~Li and X.~Yuan.
\newblock The augmented lagrangian method with full jacobian decomposition and
  logarithmic-quadratic proximal regularization for multiple-block separable
  convex programming.
\newblock {\em Preprint}, 2015.

\bibitem{Liavas}
A.P. Liavas and N.D. Sidiropoulos.
\newblock Parallel algorithms for constrained tensor factorization via the
  alternating direction method of multipliers.
\newblock {\em Arxiv Preprint: https://arxiv.org/abs/1409.2383}, 2014.

\bibitem{Lin}
T.~Lin, S.~Ma, and S.~Zhang.
\newblock An extragradient-based alternating direction method for convex
  minimization.
\newblock {\em Found. Comput. Math.}, pages 1--25, 2015.

\bibitem{Lin2015Siam}
T.~Lin, S.~Ma, and S.~Zhang.
\newblock On the global linear convergence of the admm with multiblock
  variables.
\newblock {\em SIAM J. Optim.}, 25(3):1478--1497, 2015.

\bibitem{Lin2015ch}
T.~Lin, S.~Ma, and S.~Zhang.
\newblock On the sublinear convergence rate of multi-block admm.
\newblock {\em J. Oper. Res. China}, 3(3):251--274, 2015.

\bibitem{Lin2016}
T.~Lin, S.~Ma, and S.~Zhang.
\newblock Iteration complexity analysis of multi-block admm for a family of
  convex minimization without strong convexity.
\newblock {\em J. Sci. Comput.}, 69(1):52--81, 2016.

\bibitem{JacobiADMM1}
Q.~Liu, X.~Shen, and Y.~Gu.
\newblock Linearized admm for non-convex non-smooth optimization with
  convergence analysis.
\newblock {\em Arxiv Preprint: https://arxiv.org/abs/1705.02502v1}, 2017.

\bibitem{MBProxNonconvRenJeff}
J.~G. Melo and R.~D.~C. Monteiro.
\newblock Iteration-complexity of a linearized proximal multiblock admm class
  for linearly constrained nonconvex optimization problems.
\newblock {\em Available on: http://www.optimization-online.org}, 2017.

\bibitem{monteiro2010iteration}
R.~D.~C. Monteiro and B.~F Svaiter.
\newblock Iteration-complexity of block-decomposition algorithms and the
  alternating direction method of multipliers.
\newblock {\em SIAM J. Optim.}, 23(1):475--507, 2013.

\bibitem{Mordu2006}
B.S. Mordukhovich.
\newblock {\em Variational analysis and generalized differentiation {I}: basic
  theory}.
\newblock Grundlehren der mathematischen Wissenschaften. Springer, Berlin,,
  2006.

\bibitem{LanADMM}
Y.~Ouyang, Y.~Chen, G.~Lan, and E.~Pasiliao~Jr.
\newblock An accelerated linearized alternating direction method of
  multipliers.
\newblock {\em SIAM J. Imaging Sci.}, 8(1):644--681, 2015.

\bibitem{VariaAna}
R.~T. Rockafellar and R.~J.-B. Wets.
\newblock {\em Variational analysis}.
\newblock Springer, Berlin, 1998.

\bibitem{Multi-blockBregman}
F.~Wang, W.~Cao, and Z.~Xu.
\newblock Convergence of multi-block {B}regman {A}{D}{M}{M} for nonconvex
  composite problems.
\newblock {\em Arxiv preprint: https://arxiv.org/abs/1505.03063}, 2015.

\bibitem{BregmanADMM}
F.~Wang, Z.~Xu, and H.~K. Xu.
\newblock Convergence of {B}regman alternating direction method with
  multipliers for nonconvex composite problems.
\newblock {\em Arxiv preprint:https://arxiv.org/abs/1410.8625}, 2014.

\bibitem{paralleladm}
H.~Wang, A.~Banerjee, and Z-Q. Luo.
\newblock Parallel direction method of multipliers.
\newblock {\em Arxiv Preprint: https://arxiv.org/abs/1406.4064}, 2014.

\bibitem{wotao2015}
W.~Wang, Y.~Yin and J.~Zeng.
\newblock Global convergence of {A}{D}{M}{M} in nonconvex nonsmooth
  optimization.
\newblock {\em Arxiv preprint: https://arxiv.org/abs/1511.06324}, 2015.

\bibitem{Wen2013}
Z.~Wen, X.~Peng, X.~Liu, X.~Sun, and X.~Bais.
\newblock Asset allocation under the {B}asel accord risk measures.
\newblock {\em Arxiv preprint: https://arxiv.org/abs/1308.1321}, 2013.

\bibitem{Xu2012}
Y.~Xu, W.~Yin, Z.~Wen, and Y.~Zhang.
\newblock An alternating direction algorithm for matrix completion with
  nonnegative factors.
\newblock {\em Frontiers Math. China}, 7(2):365--384, 2012.

\bibitem{Stepzisenoncon}
L.~Yang, T.~K. Pong, and X.~Chen.
\newblock Alternating direction method of multipliers for a class of nonconvex
  and nonsmooth problems with applications to background/foreground extraction.
\newblock {\em SIAM J. Imaging Sci.}, 10(1):74--110, 2017.

\bibitem{Zhang2014}
R.~Zhang and J.~T. Kwok.
\newblock Asynchronous distributed admm for consensus optimization.
\newblock {\em Proceedings of the 31st International Conference on Machine
  Learning}, 2014.

\end{thebibliography}
\end{document}